\documentclass{amsart}

\usepackage{amssymb}

\usepackage[all]{xy}

\usepackage{listings}

\usepackage{color}

\newtheorem{thm}{Theorem}

\newtheorem{lem}[thm]{Lemma}
\newtheorem{cor}[thm]{Corollary}

\newtheorem{prop}[thm]{Proposition}

\theoremstyle{definition}
\newtheorem{defn}[thm]{Definition}

\newtheorem{exmp}[thm]{Example}

\newtheorem{ques}[thm]{Question}    

\newtheorem{rem}[thm]{Remark}          

\newtheorem*{ack}{Acknowledgments}      

\newtheorem{defn-thm}[thm]{Definition--Theorem}  
\newtheorem{defn-lem}[thm]{Definition--Lemma}  

\theoremstyle{remark}

\renewcommand{\c}[0]{{\mathbb C}}  

\renewcommand{\o}[0]{{\mathcal O}} 
\newcommand{\z}[0]{{\mathbb Z}}

\newcommand{\p}[0]{{\mathbb P}}

\newcommand{\pic}[0]{\operatorname{Pic}}

\newcommand{\rank}[0]{\operatorname{rank}}

\newcommand{\sF}{\mathcal{F}}

\def\loccoh#1.#2.#3.#4.{H^{#1}_{#2}(#3,#4)}

\DeclareMathAlphabet{\mathchanc}{OT1}{pzc}%
                                {m}{it}

\begin{document}
\bibliographystyle{amsalpha}


\title[Ulrich line bundles on Enriques surfaces of degree four]{Ulrich line bundles on Enriques surfaces with a polarization of degree four}
\author{Marian Aprodu}
\address{Faculty of Mathematics and Computer Science, University of Bucharest, 14 Academiei Street, 010014 Bucharest, Romania}
\email{marian.aprodu@fmi.unibuc.ro}
\address{Simion Stoilow Institute of Mathematics of the Romanian Academy, P.O. Box 1-764, 014700 Bucharest, Romania}
\email{marian.aprodu@imar.ro}
\author{Yeongrak Kim}
\address{Max Planck Institut f\"ur Mathematik, Vivatsgasse 7, 53111 Bonn, Germany}
\email{yeongrakkim@mpim-bonn.mpg.de}
\dedicatory{To the memory of Professor Alexandru Lascu}

\begin{abstract}
In this paper, we prove the existence of an Enriques surface with a polarization of degree four with an Ulrich bundle of rank one. As a consequence, we prove that general polarized Enriques surfaces of degree four, with the same numerical polarization class, carry Ulrich line bundles. 
\end{abstract}

\maketitle


\section{Introduction}

Let $X\subset \mathbb P^N$ be an $n$--dimensional smooth projective variety
and put $H=\mathcal{O}_X(1)$. An {\em Ulrich bundle on $X$} 
(with respect to the given embedding) 
\cite{ESW03} is a vector bundle whose twists satisfy a set of vanishing conditions on cohomology:
\[
H^i (X, E(-iH))=0 \text{ for all } i>0
\]
and
\[
H^j (X, E(-(j+1)H)=0 \text { for all } j<n.
\]
The presence of twists in the definition shows that this notion strongly depends on the embedding in the projective space. The definition makes sense also for an irreducible variety $X$, not necessarily smooth \cite{ESW03}.

Ulrich bundles were introduced in commutative algebra in relation to maximally-generated maximal Cohen-Macaulay modules \cite{Ul84}.  They made a spectacular appearance in algebraic geometry in recent works of Beauville and Eisenbud-Schreyer and their importance is motivated by the relations with the Cayley-Chow forms \cite{Bea00}, \cite{ESW03} and with the cohomology tables \cite{ES11}.  

Let us briefly recall the theory of cohomology tables and how Ulrich bundles appear naturally in 
this context. For any coherent sheaf $\sF$ on the variety $X$, the cohomology table $CT(\sF)$ of $\sF$ is defined as the table 

\begin{center}
\begin{tabular}[c]{c c c c c}
\\
\hline
$\cdots$ & $\gamma_{n,-n-1}$ & $\gamma_{n, -n}$ & $\gamma_{n, -n+1}$ & $\cdots$ \\
$\cdots$ & $\gamma_{n-1,-n}$ & $\gamma_{n-1, -n+1}$ & $\gamma_{n-1, -n+2}$ & $\cdots$ \\
$\cdots$ & $\vdots$ & $\vdots$ & $\vdots$ & $\cdots$ \\
$\cdots$ & $\gamma_{0,-1}$ & $\gamma_{0,0}$ & $\gamma_{0,1}$ & $\cdots$ \\
\hline
\\
\end{tabular}
\end{center}
where $\gamma_{i,j} = h^i (X, \sF (j))$ \cite{ES11}. The cohomology table $CT(\sF)$ is an element of the 
space $\prod_{-\infty}^\infty\mathbb Q^{n+1}$. Varying the sheaf on $X$, since $CT(\sF\oplus \sF')=CT(\sF)+CT(\sF'),$
the rays of these tables define a cone, called  {\em the cone of cohomology tables} and denoted by 
$\mathcal{C}(X, \o_X(1))$. Eisenbud and Schreyer proposed a study of this cone, and obtained a nice 
description in the case of projective spaces \cite{ES11}. In the general case, we observe that a linear 
projection $\pi : X \to \p^n$ induces an injective map $\pi_{*} : \mathcal{C} (X, \o_X(1)) \to \mathcal{C}(\p^n , \o_{\p^n}(1))$. 
If $E$ is a vector bundle on $X$ such that its direct image on $\p^n$ is trivial, then $\pi^{*}( \cdot ) \otimes E$ is an inverse of $\pi_{*}$, and hence $\pi_{*}$ becomes an isomorphism. An application of the Beilinson spectral sequence  and of the Leray spectral sequence for the finite map $\pi$ implies that $\pi_*E$ is trivial if and only if $E$ is Ulrich. One concludes that the cone of cohomology table of $X$ is the same with the cone of the $n$--dimensional projective space if and only if $X$ carries and Ulrich bundle \cite{ES11} and hence the existence problem of Ulrich bundles becomes very important.
From the view-point of the cone of cohomology tables, the rank plays no role in the existence problem, in practice, we try to find Ulrich bundles of the smallest rank possible.

If the given polarization $H$ is a multiple of another polarization $H'$ then the existence of $H'$-Ulrich bundles implies the existence of $H$-Ulrich bundles of much larger rank, \cite[Proposition 5.4]{ESW03}. This phenomenon justifies a straightforward extension of the definition to take into account also polarizations which are not very ample, see Definition \ref{defn:Ulrich}. From the cone of cohomology tables view-point, this generalization is a perfectly legitimate action. It has however some deficiencies, one of them being the possible lack of a geometric interpretation of the existence of Ulrich bundles for polarizations which are not very ample. We give one example here. In rank-two, Eisenbud and Schreyer proposed the notion of {\em special} Ulrich bundles, which are Ulrich bundles on a surface $X$, of determinant $\o_X(K_X+3H)$ ($H$ is considered very ample). There is a prominent merit of the existence of special Ulrich bundles. Via \cite[Corollary 3.4]{ESW03}, a special Ulrich bundle provides a very nice presentation of the Cayley-Chow form of $X$. Indeed, $X$ admits a Pfaffian B{\'e}zout form in Pl{\"ucker} coordinates. If the polarization is not very ample, the Cayley-Chow form might not even exist.

In this short note, we investigate Ulrich bundles on Enriques surfaces with a polarization of degree four. Note that since a degree-4 polarization gives a $4:1$ map to $\p^2$, it is obviously not very ample, and hence our setup should be interpreted in the extended context of ample (not very ample) polarizations. We prove that there are Enriques surfaces with polarizations of degree four which carry Ulrich line bundles. We denote by $\sF_5$ the moduli space of polarized $K3$ surfaces of degree 8. The locus
\[
\mathcal{NL}_{7,12}:= \{ (X, H_X) \in\mathcal F_5\ \vert \ \exists \ M \in \pic(X) \text{ with } H_X \cdot M = 12, M^2 = 12 \} \subset \sF_{5}
\]
is an irreducible component of the Noether-Lefschetz locus in $\mathcal F_5$ and
\[
\mathcal{U} := \{ (X, H_X) \ \vert \ \exists \ H_X \text{-Ulrich line bundle } M \}
\]
is an open subset in $\mathcal{NL}_{7,12}$. The locus of polarized $K3$ surfaces which cover Enriques surfaces can be described as
\[
\mathcal{K} := \{ (X, H_X) \ \vert \ \exists \ \theta : X \to X \text{ fixed-point-free involution such that } H_X \simeq \theta^* H_X \}.
\]
In the main result, Theorem \ref{Thm:ExistenceOfUlrichLineBundles}, we show that the intersection $\mathcal U\cap\mathcal K$ is non-empty. The proof is completed with the help of the Macaulay2 computer-algebra system. Moreover, the line bundle that we construct is a pullback of a line bundle from the Enriques surface, which turns out to be Ulrich, too. 

The outline of the paper is the following. In section \ref{sec:prel} we recall a few facts related to Ulrich bundles and on the geometry of Enriques surfaces. In section \ref{sec:UlrichonEnriques} we prove the existence of an Enriques surface with a polarization of degree four with an Ulrich bundle of rank one. As already mentioned, the construction uses the existence of an Ulrich line bundle on its $K3$ cover. As a consequence, we prove that a general polarized Enriques surface of degree four with the same numerical polarization carries an Ulrich line bundle, Corollary \ref{cor:general}.

\section{Preliminaries}
\label{sec:prel}

\subsection{Ulrich bundles} In this section we briefly review the definitions and properties of Ulrich bundles. We extend slightly the setup, to allow polarizations which are not very ample:

\begin{defn}[compare to \cite{ESW03}, Proposition 2.1]
\label{defn:Ulrich}
Let $X$ be a projective variety of dimension $n$ and $H$ be an ample and globally generated line bundle on $X$.  We say that a vector bundle $E$ on $X$ is $H$-\emph{Ulrich} (or \emph{Ulrich with respect to $H$}) if
\[
H^i (X, E(-iH))=0 \text{ for all } i>0
\]
and
\[
H^j (X, E(-(j+1)H)=0 \text { for all } j<n.
\]
\end{defn}

\begin{rem}
\label{rmk:pushforward}
With this definition, if $Y$ is the image of $X$ in $\mathbb PV^*$ via the morphism $\varphi$ given by a linear system corresponding to a space $V\subset H^0(X,H)$ which generates $H$, a bundle $E$ is Ulrich with respect to $H$ if and only if $\varphi_*E$ is Ulrich on $Y\subset \mathbb PV^*$.
\end{rem}

In \cite{ESW03}, the original definition assumes that the polarization is very ample. The potential of this extended definition is underlined by the following result, see \cite[Proposition 5.4 and Corollary 5.7]{ESW03}

\begin{prop}\label{prop:UlrichForduple}
Let $\varphi : X \to \p^n$ be a finite morphism and denote $H_X=\varphi^* \o_{\p^n}(1)$.
If $X$ carries an Ulrich bundle $E$ with respect to $H_X$, then $X$ carries an Ulrich bundle with respect to $d H_X$ for any integer $d>0$.
\end{prop}

The existence of Ulrich bundles with respect to multiples of $H_X$ is hence ensured by the existence of $H_X$-Ulrich bundles, however, the ranks might differ drastically.

Most of the cases known to carry Ulrich bundles in the classical definition continue to have Ulrich bundles also in this extended framework. We discuss below, in Example \ref{ex:curve}, the curve case which is identical with \cite{ESW03}. 

\begin{exmp}\label{ex:curve}
 If $X$ is a curve of genus $g$, $H$ is an ample and globally generated line bundle on $C$, and $L$ is an arbitrary 
 line bundle, then $L$ is $H$-Ulrich if and only if $\mathrm{deg}(L-H)=g-1$ and $h^0(L-H)=0$. Hence a
 general line bundle of degree $\mathrm{deg}(H)+g-1$ will be $H$-Ulrich.
\end{exmp}

In the sequel, we work on a projective surface $S$.

\begin{defn}[\cite{ESW03}]
Let $S$ be a projective surface and $H$ be an ample and globally generated line bundle on $S$. A vector bundle $E$ of rank 2 is called \emph{special Ulrich} if it is $0$-regular with respect to $H$ and $\det E = \o_S(K_S + 3H)$. 
\end{defn}

One can check immediately that a special Ulrich bundle is Ulrich. Also note that any $H$-Ulrich bundle $E$ on $S$ satisfies
\[
H \cdot \left( c_1(E) - \frac{\rank(E)}{2}(K_S + 3H)\right) = 0
\]
since $\chi(E(-H)) = \chi(E(-2H))=0$, \cite{AFO12}. Hence, special Ulrich bundles are the simplest vector bundles of rank 2 which satisfy the above identity. 
Eisenbud and Schreyer proved the following structure result:
\begin{prop}\cite[Proposition 6.2]{ESW03}
\label{prop:ESW}
Let $C \in |K_S + 3H|$ be a smooth curve on $S$ and let $A$ be a line bundle on $C$ with 
\[
\deg A = \frac{5}{2}H^2 + \frac{3}{2}(K_S \cdot H) + 2 \chi(\o_S).
\]
If $\sigma_0, \sigma_1 \in H^0 (A)$ define a base point free pencil and $H^1(C, A(K_S + H)) = 0$, then the bundle $E$ defined by the Lazarsfeld-Mukai sequence
\[
0 \to E^{\vee} \to \o_S^{\oplus 2} \stackrel{(\sigma_0, \sigma_1)} \longrightarrow A \to 0
\]
is a special Ulrich bundle. Conversely, every special rank 2 Ulrich bundle on $S$ can be obtained from a Lazarsfeld-Mukai sequence. 
\end{prop}

The bundles $E$ from the proposition are called {\em Lazarsfeld-Mukai} bundles. They have
been defined first on $K3$ surfaces \cite{La86, Mu89} and they are connected with several 
distinct problems involving curves on $K3$ surfaces: Brill-Noether theory, classification of Fano
varieties, syzygies etc. They are very natural and interesting objects with applications in several problems.
In our specific situation, we see that the Lazarsfeld-Mukai bundles with given Chern classes are the only candidates 
to be special Ulrich, in view of Proposition \ref{prop:ESW}.

\begin{exmp}\label{ex:K3}
 Assume $S$ is a $K3$ surface and $H_S$ be a very ample line bundle on $S$. In \cite[Theorem 0.4]{AFO12}, the existence of special Ulrich bundles on $K3$ surfaces satisfying a mild Brill-Noether condition 
 is proved. Specifically, it is required that the Clifford index of a
 general cubic section be computed by $H_S$. A $K3$ surface whose
 Picard group is generated by $H_S$ automatically satisfies this
 hypothesis. In \cite{AFO12}, $H_S$  was supposed to be very
 ample. However, the exactly same proof goes through even if we only assume that $H_S$ is ample and globally generated.

As noted in \cite{AFO12}, the sufficient Brill-Noether condition on $K3$ surfaces is used only to ensure the existence of a base-point-free pencil of degree $\frac{5}{2}H_S^2+4$ on the cubic sections. However, there are cases not covered by this Brill-Noether condition and which still carry Ulrich bundles, and even special Ulrich bundles.
\end{exmp}

\begin{exmp}\label{ex:Enriques}
 If $S$ is an Enriques surface, and $H_S$ is ample and globally generated, the existence of special Ulrich bundles on $S$ was proved in \cite{Bea16}. In loc.cit. it was assumed that $H_S$ is very ample, however, the proof goes through even under weaker assumptions.
In some cases, $S$ carries Ulrich line bundles \cite{BN16}. Borisov and Nuer conjectured that this should always be the case, for very ample polarizations on unnodal Enriques surfaces.
\end{exmp}

\section{Construction of Ulrich bundles using $K3$ covers}
\label{sec:UlrichonEnriques}

In this section we describe how we obtain an Ulrich bundle on an Enriques surface from its $K3$ cover. Let $Y$ be an Enriques surface and $H_Y$ be an ample and globally generated line bundle on $Y$. It admits an {\'e}tale  $K3$ cover, namely $\sigma : X \to Y$ such that $X$ is a $K3$ surface with a fixed-point-free involution $\theta : X \to X$ which induces $\sigma : X \to X/\theta \simeq Y$. 

Recall from Remark \ref{rmk:pushforward} that if there is an Ulrich bundle $E$ on $X$ with respect to $H_X := \sigma^* H_Y$, 
then its push-forward $F = \sigma_{*}E$ is an Ulrich bundle of rank $2 \cdot rk(E)$ on $(Y, H_Y)$. The main goal of this section is to construct an Ulrich line bundle on a particular $(Y, H_Y)$ occurring as a direct summand of the push-forward of an Ulrich line bundle $M$ on $X$ by $\sigma$.

It is natural to ask which polarized $K3$ surfaces $(X, H_X)$ carry an $H_X$-Ulrich line bundle $M$ equipped with a fixed-point-free involution $\theta$. We divide into smaller questions as follows:

\begin{ques} \
\begin{enumerate}
\item Which $K3$ surface $X$ can cover an Enriques surface $Y$?
\item Given such a covering $\sigma : X \to Y$, which $H_X$ can be described as the pull-back of an ample and globally generated line bundle $H_Y$ on $Y$?
\item Which polarized $K3$ surfaces $(X, H_X)$ carry Ulrich line bundles?
\end{enumerate}
\end{ques}

There is a very nice criterion in \cite{Keu90} which answers the first question. It gives a whole class of $K3$ covers. In this paper, we only use a weaker result:

\begin{thm}\cite[Theorem 2]{Keu90}
Every algebraic Kummer surface is the $K3$-cover of some Enriques surface.
\end{thm}

Horikawa's theorem answers the second question. We follow the notation in \cite{Keu90}. Let $\Lambda = U \oplus U \oplus U \oplus E_8 (-1) \oplus E_8 (-1)$ be the $K3$ lattice. We choose a basis of $\Lambda$ of the form $v_1, v_2, v_1^{\prime}, v_2^{\prime}, v_1^{\prime \prime}, v_2^{\prime \prime}, e_1^{\prime}, \cdots, e_8^{\prime}, e_1^{\prime \prime}, \cdots, e_8^{\prime \prime}$ where the first 3 pairs are the standard bases of $U$ and the remaining 2 octuples are the standard bases of $E_8(-1)$. There is an involution $\vartheta : \Lambda \to \Lambda$  given by
\[
\vartheta(v_i) = -v_i, \vartheta(v_i^{\prime}) = v_i^{\prime \prime}, \vartheta(v_i^{\prime \prime}) = v_i^{\prime}, \vartheta(e_i^{\prime}) = e_i^{\prime \prime}, \vartheta(e_i^{\prime \prime}) = e_i^{\prime}.
\]
We denote the $\vartheta$-invariant sublattice by $\Lambda^{+} \simeq U(2) \oplus E_8 (-2)$. Recall that the covering map $\sigma : X \to Y$ is determined by the choice of a fixed-point-free involution $\theta : X \to X$. 

\begin{thm} \cite[I, Theorem 5.4]{Hor78}
There is an isometry $\phi : H^2 (X, \z) \to \Lambda$ such that the following diagram

\[
\xymatrix{ H^2 (X, \z) \ar[r]^{\theta^*} \ar[d]_\phi & H^2(X, \z) \ar[d]^\phi \\
\Lambda \ar[r]^{\vartheta} & \Lambda
}
\]
commutes. In particular, $\phi$ induces an isomorphism
\[
\bar{\phi} : H^2(X, \z)^{\theta^*} = \sigma^* H^2(Y, \z) = \sigma^* \pic(Y) \to \Lambda^{+}.
\]
\end{thm}

Hence, Horikawa's theorem implies that a divisor (equivalently, a line bundle) which is invariant under $\theta$ can be obtained by the pull-back of a divisor on $Y$, and vice versa.

For the last question, there are some numerical conditions which filter out most of line bundles. Let $X$ be a $K3$ surface and $H_X$ be an ample and globally generated line bundle with $H_X^2=2s>0$. If there is an Ulrich line bundle $M$ with respect to $H_X$, it must satisfy 
$\chi(M - H_X) = \chi(M - 2H_X) = 0$. Using the Riemann-Roch formula we have
\[
H_X \cdot \left(\frac{3}{2} H_X - M \right) = 3s - (H_X \cdot M) = 0,
\]
so $(H_X \cdot M) = 3s$. Since $M$ is Ulrich, applying Riemann-Roch formula once again yields
\[
2 + \frac{M^2}{2} = \chi(M) = h^0 (M) = \deg(X) \cdot \rank(M) = 2s,
\]
so $M^2 = 4s-4$. Note that not all line bundles $M$ with $(M \cdot H_X) = 3s$, $M^2 = 4s-4$ are Ulrich, see Remark \ref{Rem:NonUlrichLineBundle}.

\begin{lem}\label{Lemma:UlrichLineBundleCondition}

Let $X$ be a $K3$ surface, and $H_X$ be an ample and globally generated line bundle with $H_X^2 = 2s>0$. 
Let $M$ be a line bundle on $X$ with $H _X\cdot M = 3s$ and $M^2 = 4s-4$. $M$ is an Ulrich line bundle with respect to $H_X$ if and only if both corresponding divisors $M - H_X$ and $2H_X - M$ are not effective.
\end{lem}
\begin{proof}
It is enough to show that the 4 cohomology groups 
\[
H^0(X, M(-H_X)), H^1(X, M(-H_X)), H^1(X, M(-2H_X)), H^2(X,M(-2H_X))
\]
 vanish simultaneously. By Riemann-Roch formula, $\chi(M-H_X) = 2 + \frac{1}{2} (M - H_X)^2 = 0$ and similarly $\chi(M-2H_X) = 0$. Since $H_X$ is ample and $H_X \cdot (H_X - M) = -s < 0$, $H_X - M$ cannot be effective, that is, $h^0(X,H_X - M) = h^2 (X,M - H_X) = 0$. Similarly, we see that $M-2H_X$ is not effective, so $h^0(X,M-2H_X) = 0$. So we have 2 equalities
\begin{eqnarray*}
h^0(X, M-H_X) & = & h^1 (X, M-H_X) \\
h^1(X, M-2H_X) & = & h^2 (X, M- 2H_X).
\end{eqnarray*}
Since $h^2(X, M-2H_X) = h^0 (X, 2H_X - M)$, we get the desired result.
\end{proof}

Before constructing an $H_X$-Ulrich line bundle on some $K3$ surface $X$ which covers an Enriques surface, we briefly explain why this problem is quite difficult. We denote by $\sF_{s+1}$ the moduli space of polarized $K3$ surfaces $\sF_{s+1}$ of degree $2s$. The \emph{Noether-Lefschetz locus}, defined as 
\[
\mathcal{NL}:= \{ (X, H_X) \ \vert \ \mbox{rk} (\pic(X))\ge 2\}
\]
is a countable union of divisors inside $\sF_{s+1}$.
When we fix the number $H_X^2 = 2s$, the locus
\[
\mathcal{NL}_{2s-1,3s}:= \{ (X, H_X) \ \vert \ \exists \ M \in \pic(X) \text{ with } H_X \cdot M = 3s, M^2 = 4s-4 \} \subset \sF_{s+1}
\]
is an irreducible component of $\mathcal{NL}$ (the subscript $2s-1$ stands for $\frac{1}{2}M^2+1$). Lemma \ref{Lemma:UlrichLineBundleCondition} and the semicontinuity of the Ulrich condition in flat families imply that the locus
\[
\mathcal{U} := \{ (X, H_X) \ \vert \ \exists \ H_X \text{-Ulrich line bundle } M \}
\]
is an open subset in $\mathcal{NL}_{2s-1,3s}$.  Note also that the locus of polarized $K3$ surfaces which cover Enriques surfaces can be described as
\[
\mathcal{K} := \{ (X, H_X) \ \vert \ \exists \ \theta : X \to X \text{ fixed-point-free involution such that } H_X \simeq \theta^* H_X \}
\]
which is a closed subset of large codimension (since the Picard number of $X$ is at least 10) in the moduli space of polarized $K3$ surfaces. Hence the problem reduces to finding one element which lies both in an open subset of a Noether-Lefschetz divisor and in a subvariety of large codimension of $\mathcal F_{s+1}$.

\medskip

However, for $s=4$, we are able to prove that the intersection is nonempty by constructing an explicit example of a $K3$ cover $X$. 


\begin{thm}
\label{Thm:ExistenceOfUlrichLineBundles}
When $s=4$, the intersection $\mathcal{U} \cap \mathcal{K} \subset \sF_5$ is nonempty, that is, there is a polarized $K3$ surface $(X, H_X)$ with $H_X^2 = 8$ which is a $K3$-cover of an Enriques surface $\sigma : X \to Y$ and carries an $(H_X=\sigma^* H_Y )$-Ulrich line bundle $M$ for some ample line bundle $H_Y$ on $Y$. Moreover, $M$ can be chosen to be the pull-back of an $H_Y$-Ulrich line bundle on $Y$.
\end{thm}

\begin{proof}
We proceed in two steps. In the first step, we 
place ourselves in a more general setup that permits the construction of a class of examples.
In the second step we find an explicit example, using Macaulay2.

\medskip

\emph{Step 1.} Our candidate $K3$ covers are Kummer surfaces $X$ associated to Jacobian abelian surfaces together with suitable polarizations and potential Ulrich line bundles.
Let ${C}$ be a general curve of genus 2, write $C\to \mathbb P^1$ as a double cover of the projective line and denote by $p_1,\ldots,p_6$ the Weierstrass points. They define sixteen theta--characteristics:
\[
[p_i],\ i=1,\ldots,6,\mbox{ the odd theta--characteristics, and}
\]
\[
[p_i+p_j-p_k],\ i,j,k = 1,\ldots,6\mbox{ mutually distinct, the even theta--characteristics}.
\]

The Jacobian $\mathcal A = J( C)$ is an Abelian surface with N{\'e}ron-Severi group $NS(\mathcal A) = \z \cdot [\Theta]$ with $\Theta^2 = 2$. 

The complete linear system $|2 \Theta|$ defines a morphism to $\p^3$ and it factors through the singular surface $\mathcal A/\iota$, where $\iota$ is the involution on $\mathcal A$ with 16 fixed points. This embeds $\mathcal A /\iota$ as a quartic hypersurface in $\p^3$ with 16 nodes. The Kummer surface $X = Km(\mathcal A)$ associated to $\mathcal A$ is the minimal desingularization of $\mathcal A/\iota$.  Let $L \in \pic(X)$ be a line bundle induced by the hyperplane section of the quartic surface $\mathcal A /\iota$, and let $E_1, \ldots, E_{16}$ be the 16 exceptional $(-2)$-curves on $X$ arising from the nodes of $\mathcal A /\iota$. By abusing the notation, the curves $E_i$ are usually called in literature \textit{nodes}, too.
We have $L^2 = 4, L \cdot E_i = 0$ and $E_i \cdot E_j = -2\delta_{ij} $. 

Beside the set of the nodes mentioned above, there is another set of sixteen $(-2)$--curves, called \emph{tropes} constructed from theta--characteristics, see, for example \cite[pag. 175]{Oha09}. Together with the nodes, they form a $(16)_6$ configuration.


We take an ample line bundle $H_X = 2L - \frac{1}{2} \sum_{i=1}^{16} E_i$. Note that $H_X$ induces a smooth projective model of $X$ as the complete intersection of 3 quadrics in $\p^5$ \cite[Theorem 2.5]{Shi77}, \cite[Section 5.1]{GS13}. By choosing suitable coordinates, we may write $C$ as
\[
y^2 = \prod_{j=0}^{5} (x - s_j)
\]
for some 6-tuple of pairwise distinct numbers $s_j \in \c$. Then its projective model $\varphi_{H} : X \hookrightarrow \p^5$ is defined by the equations
\[
{\setlength\arraycolsep{2pt}
\left\{ \begin{array}{rcrcrcrcrcrc}
z_0^2 & + & z_1 ^2 & + & z_2^2 & + & z_3^2 & + & z_4^2 & + & z_5^2 & =0 \\
s_0z_0^2 & + & s_1z_1 ^2 & + & s_2z_2^2 & + & s_3z_3^2 & + & s_4z_4^2 & + & s_5z_5^2 & =0 \\
s_0^2 z_0^2 & + & s_1^2 z_1 ^2 & + & s_2^2 z_2^2 & + & s_3^2 z_3^2 & + & s_4^2 z_4^2 & + & s_5^2 z_5^2 & =0 \\
\end{array}
\right.}
\]
in $\p^5$ \cite[Theorem 2.5]{Shi77}. Note that there are ten fixed-point-free involutions given by changing the sign of three coordinates, for example
\[
\theta : (z_0, z_1, z_2, z_3, z_4, z_5) \mapsto (-z_0, -z_1, -z_2, z_3, z_4, z_5),
\] 
and these involutions correspond to the ten even theta-characteristics \cite[p. 233]{Mu12}, \cite[p. 166]{Oha09}. Reordering the Weierstrass points if necessary, we may assume that the particular involution $\theta$  which  changes the signs of the first three coordinates corresponds to the theta--characteristic $\beta = [p_4+p_5-p_6]$.

The involution $\theta$ defined above induces the covering map over an Enriques surface $\sigma : X \to Y = X/\theta$. We can easily check that $H_X$ is $\theta$-invariant as follows. Note that $H_X$ can be represented as a hyperplane divisor of $X \subset \p^5$. For instance, we take the hyperplane section $Z:= \{z_0 = 0\} \cap \varphi_{H_X}(X) \subset \p^5$ and see immediately that $Z$ is $\theta$-invariant. In other words, $c_1(H_X)$ lies on $\theta$-invariant lattice $\Lambda^{+} \subset H^2(X, \z)$. Therefore, Horikawa's theorem implies $H_X = \sigma^* H_Y$ for some line bundle $H_Y$ on~$Y$.

Since $H_X = \sigma^* H_Y$ is ample, we see that $H_Y$ is also ample by Nakai-Moishezon criterion. By Riemann-Roch formula, we have $h^0 (Y, H_Y) = h^0 (Y, K_Y + H_Y) = 3$, which means that $H_Y$ gives rise to a 4-fold cover of $\p^2$.  

Following \cite{Oha09}, we relabel the nodes by the corresponding 2--torsion points in $\mathcal A$: 
\begin{eqnarray*}
E_0 & = & \text{ node corresponding to } [0] \in \mathcal{A}; \\
E_{ij}=E_{[p_i - p_j]} & = & \text{ node corresponding to } [p_i - p_j] \in \mathcal{A}, 1 \le i < j \le 6.
\end{eqnarray*}
The tropes are labelled using their associated theta--characteristics \cite{Oha09}, e.g. $T_i=T_{[p_i]}$ corresponds to $[p_i]$ and $T_{ijk}=T_{[p_i+p_j-p_k]}$ corresponds to $[p_i+p_j-p_k]$ for any $i<j<k$. Obviously, if $\{i,j,k\}\cup \{\ell, m,n\}=\{1,\ldots,6\}$ then $T_{ijk}=T_{\ell m n}$.

Since the fixed-point-free involution $\theta$ is a ``switch'' induced by the even theta characteristic $\beta = [p_4+p_5-p_6]$, it swaps the nodes $E_{\alpha}$ and the tropes $T_{\alpha+\beta}$ in the following way, \cite[Section 4, Section 5]{Oha09}:

\[
\begin{array}{ccc|ccc}
\hline
\text{Nodes} & \ & \text{Tropes} &  \text{Nodes} & \ & \text{Tropes} \\
\hline
E_0 & \leftrightarrow & T_{456}  & E_{25} & \leftrightarrow & T_{246} \\
E_{12} & \leftrightarrow & T_{3} &  E_{26} & \leftrightarrow & T_{136} \\
E_{13} & \leftrightarrow & T_{2} & E_{34} & \leftrightarrow & T_{356} \\
E_{14} & \leftrightarrow & T_{156} & E_{35} & \leftrightarrow & T_{346} \\
E_{15} & \leftrightarrow & T_{146} & E_{36} & \leftrightarrow & T_{126} \\
E_{16} & \leftrightarrow & T_{236} & E_{45} & \leftrightarrow & T_{6} \\
E_{23} & \leftrightarrow & T_{1} & E_{46} & \leftrightarrow & T_{5} \\
E_{24} & \leftrightarrow & T_{256} & E_{56} & \leftrightarrow & T_{4} \\
\hline
\end{array}
\]
\\
where the corresponding tropes are computed by (see \cite[Lemma 4.1]{Oha09})

\begin{eqnarray*}
T_i & = & \frac{1}{2} (L - E_0 - \sum_{k \neq i} E_{ik})
\end{eqnarray*}
for $1\le i\le 6$ and 
\begin{eqnarray*}
T_{ij6} & = & \frac{1}{2} (L - E_{i6} - E_{j6} - E_{ij} - E_{\ell m}-E_{mn}-E_{\ell n})
\end{eqnarray*}
for $1 \le i < j \le 5$, where $\{l,m,n \}$ is the complement of $\{i, j\}$ in $\{1, 2, 3, 4, 5\}$.

Note that, since
\[
 L = 2T_6 + E_0 + E_{16} + E_{26} + E_{36} + E_{46} + E_{56},
\] 
we obtain the formula
\begin{equation}
\label{eqn:theta*L}
\theta^*(L)=3L-E_0-\sum E_{ij}.
\end{equation}

Put $M = 3L - (E_0 + E_{16} + E_{26} + E_{36} + E_{46} + E_{56} + E_{12}+E_{13}+E_{14}+E_{15} + E_{24} + E_{35})$. A direct computation using (\ref{eqn:theta*L}) shows that $\theta^*M = L + T_6 + T_1 + T_{246}+T_{356} = M$, that is, $M$ is invariant under $\theta^*$. Hence, we conclude that $M = \sigma^* N$ for some line bundle $N$ on $Y$, and $F=\sigma_* (M) = N \oplus (N \otimes K_Y)$. 

Remark that $M \cdot H_X = M^2 = 12$. Hence, in view of Lemma \ref{Lemma:UlrichLineBundleCondition}, this particular line bundle $M$ is Ulrich if and only if the divisors $M-H_X$ and $2H_X - M$ are not effective.

\medskip 

\emph{Step 2.}
Using Macaulay2, see \cite {GS}, we provide an example of a polarized $K3$ cover as above, with $M-H_X$ and $2H_X - M$ non--effective.  We take the explicit equation for the Kummer quartic surface in $\mathbb{P}^3$ for a genus 2 curve from \cite[Section 2]{Fly93}. We also refer \cite[Section 4]{Kum08} for more analysis on nodes and tropes. Let $C$ be the hyperelliptic curve given by the equation $y^2 = (x-1)(x+1)(x-2)(x+2)(x-3)(x+3)$. The corresponding equation which gives a Kummer quartic with 16 nodes is the following, \cite[Section 4.2]{Kum08}:

\bigskip

\begin{verbatim}
Macaulay2, version 1.8.2
with packages: ConwayPolynomials, Elimination, IntegralClosure,
               LLLBases, PrimaryDecomposition, ReesAlgebra,
               TangentCone

i1 : S=ZZ/32003[X,Y,Z,W];
i2 : f=7056*X^4-2016*X^2*Y^2+144*Y^4-288*X*Y^2*Z+2888*X^2*Z^2
-196*Y^2*Z^2+56*Z^4+144*X^3*W-196*X^2*Z*W+56*X*Z^2*W-4*Z^3*W
+Y^2*W^2-4*X*Z*W^2; 
i3 : I=ideal f;
o3 : Ideal of S
\end{verbatim}

\bigskip

We can easily verify that it is a singular surface with 16 distinct nodes as follows.

\bigskip

\begin{verbatim}
i4 : NODES=ideal singularLocus Proj (S/I);
o4 : Ideal of S
i5 : codim NODES
o5 = 3
i6 : degree NODES
o6 = 16
\end{verbatim}

\bigskip

We are interested in the vanishing $H^0(2H_X - M) = H^0 (L - E_{23} - E_{25}-E_{34} - E_{45}) = 0$. To compute the cohomology $H^0$ passing by the map $\phi_{|L|} : X \to \p^3$, we need to pick 4 nodes in the image corresponding to $E_{23}, E_{25}, E_{34}, E_{45}$. Following the computations in \cite[Section 4.2]{Kum08}, we have 4 points in $\p^3$
\begin{eqnarray*}
p_{23} & = & (1:1:-2:-44) \\
p_{25} & = & (1:2:-3:-42) \\
p_{34} & = & (1:0:-4:-65) \\
p_{45} & = & (1:1:-6:-84).
\end{eqnarray*}
which are 4 nodes of the Kummer quartic $\bar{X} = V(f)$.

Let $J1$ be the ideal for 4 nodes $\{ p_{23}, p_{25},p_{34}, p_{45} \}$, and $J2$ be the ideal for complementary 12 nodes. We chose the ideal manually among the minimal prime ideals to reduce hand-written computations. For  practical reasons, we consider also some intermediate saturation processes. 

\bigskip

\begin{verbatim}
i7  : LIST=minimalPrimes NODES
o7  = {ideal (Z, Y, X), ideal (- 14238Z + W, Y, X + 3556Z),
     ----------------------------------------------------------
     ideal (- 8017Z + W, Y, X + 8001Z), ideal (- 50Z + W, Y, X
     ----------------------------------------------------------
     + Z), ideal (- 14Z + W, Y + 16000Z, X + 16001Z), ideal (-
     ----------------------------------------------------------
     14Z + W, Y - 16000Z, X + 16001Z), ideal (- 14Z + W, Y +
     ----------------------------------------------------------
     10667Z, X + 10668Z), ideal (- 14Z + W, Y - 10667Z, X +
     ----------------------------------------------------------
     10668Z), ideal (- 14Z + W, Y + 5334Z, X + 5334Z), ideal (-
     ----------------------------------------------------------
     14Z + W, Y - 5334Z, X + 5334Z), ideal (- 2Z + W, Y -
     ----------------------------------------------------------
     10669Z, X - 10668Z), ideal (- 2Z + W, Y + 10669Z, X -
     ----------------------------------------------------------
     10668Z), ideal (10Z + W, Y - 5333Z, X - 5334Z), ideal (10Z
     ----------------------------------------------------------
     + W, Y + 5333Z, X - 5334Z), ideal (- 22Z + W, Y + 16001Z,
     ----------------------------------------------------------
     X - 16001Z), ideal (- 22Z + W, Y - 16001Z, X - 16001Z)}
o7  : List
i8  : Ip23=LIST_15
o8  = ideal (- 22Z + W, Y - 16001Z, X - 16001Z)
o8  : Ideal of S
i9  : Ip25=LIST_5
o9  = ideal (- 14Z + W, Y - 16000Z, X + 16001Z)
o9  : Ideal of S
i10 : Ip34=LIST_2
o10 = ideal (- 8017Z + W, Y, X + 8001Z)
o10 : Ideal of S
i11 : Ip45=LIST_8
o11 = ideal (- 14Z + W, Y + 5334Z, X + 5334Z)
o11 : Ideal of S
i12 : J1=saturate(Ip23*Ip25*Ip34*Ip45);
o12 : Ideal of S
i13 : Temp1=saturate(LIST_0*LIST_1*LIST_3*LIST_4*LIST_6*LIST_7);
o13 : Ideal of S
i14 : Temp2=saturate(LIST_9*LIST_10*LIST_11*LIST_12*LIST_13*LIST_14);
o14 : Ideal of S
i15 : J2=saturate(Temp1*Temp2);
o15 : Ideal of S
\end{verbatim}

\bigskip

Now, the element in $|2H_X - M| = |L - E_{23} - E_{25} - E_{34} - E_{45}|$ corresponds to a hyperplane section passing through $p_{23}, p_{25}, p_{34}, p_{45}$, and we can check that there is no such a hyperplane section:

\bigskip

\begin{verbatim}
i16 : HH^0(sheaf(S^{1}**module(J1)))
o16 = 0
        ZZ
o16 : ------module
      32003
\end{verbatim}

\bigskip

Similarly, if $|M-H_X|$ is nonempty, then $|2(M-H_X) | = |2L + E_{23}+E_{25}+E_{34}+E_{45} - (E_0+E_{16}+E_{26}+E_{36}+E_{46}+E_{56}+E_{12}+E_{13}+E_{14}+E_{15}+E_{24}+E_{35})|$ is also nonempty. Note that, since $L|_{E_{ij}}\cong \mathcal O_{E_{ij}}$ and $\mathcal O_{E_{ij}}(E_{ij})\cong \mathcal O_{E_{ij}}(-2)$, we obtain immediately an isomorphism
\[
H^0(2L)\cong H^0(2L+E_{23}+E_{25}+E_{34}+E_{45})
\]
and hence we can identify $|2L + E_{23}+E_{25}+E_{34}+E_{45} - (E_0+E_{16}+E_{26}+E_{36}+E_{46}+E_{56}+E_{12}+E_{13}+E_{14}+E_{15}+E_{24}+E_{35})|$ with $|2L -(E_0+E_{16}+E_{26}+E_{36}+E_{46}+E_{56}+E_{12}+E_{13}+E_{14}+E_{15}+E_{24}+E_{35})|$. Via the map $\varphi_{|L|}$ to $\mathbb{P}^3$, an element in this linear system corresponds to a quadric hypersurface passing through the 12 complementary nodes to $\{p_{23},p_{25},p_{34},p_{45}\}$. Macaulay2 computation shows however that

\bigskip

\begin{verbatim}
i17 : HH^0(sheaf(S^{2}**module(J2)))
o17 = 0
        ZZ
o17 : ------module
      32003
\end{verbatim}

\bigskip

\noindent
i.e. there is no such a quadric section.

\medskip

\emph{Conclusion.}
For the example found in the second step, since $M$ is $H_X$--Ulrich, it follows that $F$ is $H_Y$--Ulrich, and hence the direct summand $N$ is an $H_Y$--Ulrich line bundle as well.
\end{proof}

In what follows, we prove the existence of Ulrich line bundles for general Enriques surfaces with a polarization of degree four. We fix some notation. Let $h\in U\oplus E_8(-1)$ be the numerical class of the polarization $H_Y$ constructed above, $\mathcal M_{En}^0$ be the 10--dimensional moduli space of Enriques surfaces \cite{Hor78, Nam85, GH16}, and $\mathcal M_{En,h}^0$ be the moduli space of Enriques surfaces with a polarization of type $h$, \cite{GH16}. It is also 10--dimensional and irreducible, and there is a natural forgetful morphism $\varphi:\mathcal M_{En,h}^0\to \mathcal M_{En}^0$ obtained from the descriptions of the two moduli spaces as (open subsets of) quotients of the same bounded domain, \cite[pag. 59, 61]{GH16}. We prove:

\begin{cor}
\label{cor:general}
A general polarized Enriques surface $(Y, H_Y)\in\mathcal M_{En,h}^0$ carries an $H_Y$-Ulrich line bundle.
\end{cor}

\begin{proof}
In Theorem \ref{Thm:ExistenceOfUlrichLineBundles} we constructed a polarized Enriques surface $(Y, H_Y)$ of degree four and an $H_Y$-Ulrich line bundle $N$ on it.  Note that $N$ satisfies numerical conditions $N^2 = N \cdot H_Y = 6$ since $\chi(N-H_Y) = \chi(N-2H_Y)=0$. Denote by $\eta \in U\oplus E_8(-1)$ its numerical class.

We claim that a general Enriques surface $(Y, H_Y)\in \mathcal M_{En,h}^0$ has an Ulrich line bundle. Consider the locus 
$$
\mathcal{U}_{En,h}:=\{ (Y, H_Y) \ | \ \exists \text{ an } H_Y\text{-Ulrich line bundle } N\}
$$ 
inside the moduli space $\mathcal{M}_{En,h}^0$ of polarized Enriques surfaces of degree four. Since Ulrich conditions are open in flat families, it is an open subset of the locus 
$$
\mathcal{NL}_{En,h} := \{ (Y, H_Y) \in\mathcal M_{En,h}^0 \ | \ \exists \text{ a line bundle } N \text{ such that } N^2 = N \cdot H_Y = 6 \}.
$$ 
We claim that $\mathcal{NL}_{En,h}$ coincides with the whole space $\mathcal{M}_{En,h}^0$. If it happens, then $\mathcal{U}_{En,h}$ is a nonempty open subset of $\mathcal{M}_{En,h}^0$, hence, a general polarized Enriques surface carries an Ulrich line bundle.

Note that for any Enriques surface $Y$, there exists a polarization $H$ of numerical class $h$ and a line bundle $N$ such that $H^2 = 4$, $N^2 = N \cdot H = 6$. Indeed, any line bundle $N$ of numerical class $\eta$ satisfies this condition. Via the surjective morphism $\varphi : \mathcal{M}_{En,h}^0 \to \mathcal{M}_{En}^0$, $\mathcal{NL}_{En,h}$ is dominant over $\mathcal{M}_{En}^0$ . Since $\mathcal{M}_{En,h}^0$ is an irreducible variety of dimension 10 and $\mathcal{NL}_{En,h} \subseteq \mathcal{M}_{En,h}^0$ is a closed algebraic subset which dominates $\mathcal{M}_{En}^0$ via the map $\varphi$, we conclude that $\dim \mathcal{NL}_{En,h} = 10$ and hence $\mathcal{NL}_{En,h}$ and $ \mathcal{M}_{En,h}^0$ coincide. 
\end{proof}

\begin{rem}\label{Rem:NonUlrichLineBundle}
Finding an $H_X$-Ulrich line bundle $M$ for an arbitrary $X$ is not a simple question. Indeed, there is a line bundle $M$ with $H_X \cdot M = M^2 = 12$ which is not Ulrich. Suppose that 8 exceptional curves $E_{i_1}, \ldots, E_{i_8}$ forms an \emph{even eight}, that is, $\sum_{k=1}^8 E_{i_k}$ is divisible by 2 in $\pic(X)$. Then $M = 2L - \frac{1}{2} \sum_{k=1}^8 E_{i_k}$ satisfies the numerical conditions in Lemma \ref{Lemma:UlrichLineBundleCondition}. However, we can check it directly that $M$ cannot be Ulrich. We have $M-H_X = \frac{1}{2} \sum_{k=1}^8 E_{j_k}$, where the index set $\{j_1, \ldots, j_8\}$ is the complementary set of $\{i_1, \ldots, i_8\}$ in $\{1, 2, \ldots, 16\}$. By Nikulin \cite[Corollary 5]{Nik75}, the set $\{j_1, \ldots, j_8\}$ also induces an even eight, so $M-H_X$ is effective and $M$ is not Ulrich.
\end{rem}

\begin{rem}
Corollary \ref{cor:general} can be accounted an evidence of Borisov-Nuer conjecture, even though the authors formulated it for unnodal Enriques surfaces of degree $\ge$ 10.
\end{rem}

\begin{ack}
We are indebted to the anonymous referee for their useful suggestions to improve the presentation.
The second author thanks Yongnam Lee and Alessandra Sarti for helpful discussions. The authors thank the Max Planck Institut f\"ur Mathematik in Bonn for hospitality during the preparation of this work. Marian Aprodu was partly funded by an UEFISCDI grant.
Yeongrak Kim was supported by Basic Science Research Program through the National Research Foundation of Korea funded by the Ministry of Education (NRF-2016R1A6A3A03008745). 
\end{ack}

\def\cprime{$'$} \def\cprime{$'$} \def\cprime{$'$} \def\cprime{$'$}
  \def\cprime{$'$} \def\cprime{$'$} \def\dbar{\leavevmode\hbox to
  0pt{\hskip.2ex \accent"16\hss}d} \def\cprime{$'$} \def\cprime{$'$}
  \def\polhk#1{\setbox0=\hbox{#1}{\ooalign{\hidewidth
  \lower1.5ex\hbox{`}\hidewidth\crcr\unhbox0}}} \def\cprime{$'$}
  \def\cprime{$'$} \def\cprime{$'$} \def\cprime{$'$}
  \def\polhk#1{\setbox0=\hbox{#1}{\ooalign{\hidewidth
  \lower1.5ex\hbox{`}\hidewidth\crcr\unhbox0}}} \def\cdprime{$''$}
  \def\cprime{$'$} \def\cprime{$'$} \def\cprime{$'$} \def\cprime{$'$}
\providecommand{\bysame}{\leavevmode\hbox to3em{\hrulefill}\thinspace}
\providecommand{\MR}{\relax\ifhmode\unskip\space\fi MR }
\providecommand{\MRhref}[2]{%
  \href{http://www.ams.org/mathscinet-getitem?mr=#1}{#2}
}
\providecommand{\href}[2]{#2}

\end{document}